\documentclass[a4paper,11pt,reqno]{amsart}
\usepackage{amsmath}
\usepackage{amssymb}
\usepackage{xcolor}
\usepackage{comment}
\usepackage{enumerate}
\usepackage[pdftex]{graphicx}
\usepackage{tikz}
\usetikzlibrary{intersections, calc, arrows.meta}

\usepackage{mathrsfs}
\renewcommand{\mathcal}{\mathscr}

\theoremstyle{definition} 
\newtheorem{theorem}{Theorem}[section]

\newtheorem{lemma}[theorem]{Lemma}

\newtheorem{remark}[theorem]{Remark}

\makeatletter
    
    \@addtoreset{equation}{section}
  \makeatother

\title[CLT for the elephant random walk]{Rate of moment convergence in the central limit theorem for the elephant random walk}
\thanks{M.H. is partially supported by JSPS KAKENHI Grant Numbers JP19H01791 and JP21K03272. M.T. is partially supported by JSPS KAKENHI Grant Numbers JP19H01793, JP19K03514, and JP22K03333.
}
\author{Masafumi Hayashi}
\address{Department of Mathematical Sciences, University of the Ryukyus, Okinawa, Japan}
\email{hayashi@math.u-ryukyu.ac.jp}
\author{So Oshiro}
\address{Department of Mathematical Sciences,  Graduate School of Engineering and Science, University of the Ryukyus, Okinawa, Japan}
\author{Masato Takei}
\address{Department of Applied Mathematics, Faculty of Engineering, Yokohama National University, Yokohama, Japan}
\email{takei-masato-fx@ynu.ac.jp}

\begin{document}

\begin{abstract}
The one-dimensional elephant random walk is a typical model of discrete-time random walk with step-reinforcement, and is introduced by Sch\"{u}tz and Trimper (2004). It has a parameter $\alpha \in (-1,1)$: The case $\alpha=0$ corresponds to the simple symmetric random walk, and when $\alpha>0$ (resp. $\alpha<0$), the mean displacement of the walker at time $n$ grows (resp. vanishes) like $n^{\alpha}$. The walk admits a phase transition at $\alpha=1/2$ from the diffusive behavior to the superdiffusive behavior.
In this paper, we study the rate of the moment convergence in the central limit theorem for the position of the walker when $-1 < \alpha \leq 1/2$. We find a crossover phenomenon in the rate of convergence of the $2m$-th moments with $m=2,3,\ldots$ inside the diffusive regime $-1<\alpha<1/2$. 
\end{abstract}

\maketitle

\section{Introduction}
\label{intro}

Stochastic processes with long-memory appear often in the various fields of applied mathematics and statistical mechanics. 
The elephant random walk (ERW), which is introduced by Sch\"{u}tz and Trimper \cite{SchutzTrimper04}, is known as one of the simplest such models. It is a discrete-time, nearest-neighbor random walk on integers $\mathbb{Z}$ with complete memory of its whole past.
It starts at the origin, and at time $1$ it moves to the right with probability $q \in [0,1]$ and to the left with probability $1-q$. At time $n+1$, it selects an integer $U_n$ uniformly at random from $\{1,\cdots,n\}$, and moves exactly in the same direction (resp. in the opposite direction) as that of time $U_n$ with probability $p \in (0,1)$ (resp. $1-p$).
The main interest of the study of the ERW is to investigate the long term memory effects.
In \cite{SchutzTrimper04}, it is shown that the ERW exhibits both normal and anomalous (super) diffusion, depending on the memory parameter $p$. Last few years, many authors have studied several limit theorems describing how the memory influences the asymptotic behavior of the ERW (see \cite{BaurBertoin16,Bercu18,Bertoin22Counting,Collettietal17a,Collettietal17b,FanHuMa21,KubotaTakei19JSP,MaElMachkouriFan22} and references therein).

Let $S_n$ denote the position of the ERW at time $n$.
In the case $p=q=1/2$, there is no memory effect and the ERW behaves just like the simple symmetric random walk. The results on the limiting distribution of the ERW, obtained in \cite{BaurBertoin16,Bercu18,Collettietal17a,Collettietal17b}, are summarized as follows: In the long run, the ERW with $p \in (0,3/4)$ becomes similar to the symmetric random walk in that the distribution of $S_n/\sqrt{n}$ converges to a Gaussian distribution. On the other hand, for $p \in (3/4,1)$, $S_n/n^{2p-1}$ converges to a nondegenerate random variable $L$ with probability one, where the distribution of $L$ is not Gaussian. In the critical case $p=3/4$, the distribution of $S_n/\sqrt{n \log n}$ converges to the standard Gaussian distribution. 

The central limit theorem for the diffusive regime $p\in(0,3/4)$ shows that the ERW is qualitatively similar to the ordinary random walk.
Quite recently, more quantitative forms of the central limit theorem for $p \in (0,3/4]$ are obtained by \cite{FanHuMa21,MaElMachkouriFan22}. 
Their results suggest that there is a transition of the convergence rate at $p=1/2$.
However, the results obtained in \cite{FanHuMa21,MaElMachkouriFan22} are about the change of {\it upper bounds} for the accuracy of the normal approximation.
This motivates the study in the current paper. The moment method, which is to understand the distribution of a random variable $X$ via its moments $E[X^k]$ $(k=1,2,\cdots)$, is often useful even in non-independent settings (see e.g. Section 30 of \cite{Billingsley12PMAnniversary}, or Section 2.2.3 in \cite{Tao12AMSBook}).
In this paper, we derive the {\it exact} rate of moment convergence in the central limit theorem for $\{S_n\}$ for $p \in (0,3/4]$. In particular, as a consequence of our main results, the rate of convergence of the odd order moments and of the second moment becomes small as $p$ increases, while there is a transition at $p=1/2$ about the rate of convergence of the $2m$-th moments with $m \geq 2$.

In the study of the ERW, it is often convenient to use the shifted memory parameter $\alpha:=2p-1$. The memoryless case is $\alpha=0$. The dynamics of the ERW can be rephrased as follows (see Section II in \cite{Kursten16}): If $\alpha>0$ (resp. $\alpha<0$), then with probability $\alpha$ (resp. $-\alpha$) the ERW moves in the same direction (resp. in the opposite direction) as it did in the past, and with probability $1-\alpha$ (resp. $1+\alpha$) it will go right or left with equal probability. The case $\alpha>0$ is called ``step-reinforcing", and the case $\alpha<0$ is called ``counterbalancing" (\cite{Bertoin20Counterbalancing,Bertoin21Universality}).
The main results in this paper suggest that the counterbalance does not make the convergence in the central limit theorem faster than the simple random walk, while the step-reinforcement makes it slower. This sheds some lights on the difference of the memory effects between those two cases.

The rest of the paper is organized as follows. The precise definition of the ERW, together with fundamental results, is given in Section \ref{sec:model}. Our main results are presented in Section \ref{sec:MainResults}. Our proofs are based on recursions for the moments of $S_n$ (see \eqref{eq:Oshiro-odd} and \eqref{eq:Oshiro-even} below). The proof of the result for the odd order moments in Section \ref{sec:thmHOT:OddMoments} reveals the skeleton of our strategy to obtain the exact convergence rate of moments, and is helpful for understanding the more complicated proofs needed for the even order moments, given in Sections \ref{sec:thmHOT:EvenMoments} and \ref{sec:ThmEven(iii)}.

\section{The elephant random walk and central limit theorems} \label{sec:model}

The elephant random walk is defined as follows: Let $p \in (0,1)$ and $q \in [0,1]$.
\begin{itemize}
\item The first step $X_1$ of the walker is $+1$ with probability $q$, and $-1$ with probability $1-q$.
\item For each  $n\in \mathbb{N} := \{1,2,\ldots\}$, let $U_n$ be uniformly distributed on $\{1,\ldots,n\}$, and
\begin{align*}
X_{n+1} &=  \begin{cases}
X_{U_n} &\mbox{with probability $p$}, \\
-X_{U_n} &\mbox{with probability $1-p$}. \\
\end{cases}
\end{align*} 
\end{itemize}
Each of choices in the above procedure is made independently. 
The sequence $\{X_i\}$ generates a one-dimensional random walk $\{S_n\}$ by
\[ S_0:=0,\quad \mbox{and} \quad S_n= \sum_{i=1}^n X_i \quad \mbox{for $n\in \mathbb{N}$.} \]
The parameter $p$ is referred to the memory parameter. 
When $p=1/2$,  the memory does not  affect the decision of the walker at all,  and it is essentially the symmetric random walk. 
For $p>1/2$ (resp. $p<1/2$), the walker is preferentially doing the same (resp. opposite) as in the past. 
The new parameters defined by 
\[ \alpha:=2p-1 \in (-1,1) \quad \mbox{and} \quad \beta:=E[X_1]=2q-1\in [-1,1] \]
will be convenient later. Let $\mathcal{F}_n$ be the $\sigma$-algebra generated by $X_1,\ldots,X_n$. 
For  $n \in \mathbb{N}$, the conditional distribution of $X_{n+1}$ given the history up to time $n$ is
\begin{align}
 &P(X_{n+1}= \pm 1 \mid \mathcal{F}_n) \notag \\
 &= \dfrac{\#\{i=1,\ldots,n : X_i=\pm 1\}}{n} \cdot p+  \dfrac{\#\{i=1,\ldots,n : X_i= \mp 1\}}{n} \cdot (1-p) \notag \\
& = \dfrac{n \pm S_n}{2n} \cdot p+   \dfrac{n \mp S_n}{2n} \cdot (1-p) = \dfrac{1}{2} \left( 1 \pm \alpha \cdot \dfrac{S_n}{n} \right).
\label{eq:elephantRWCondDistp}
\end{align}
The conditional expectations of $X_{n+1}$ and $S_{n+1}$ are 
\begin{align}
	 E[X_{n+1} \mid \mathcal{F}_n]&= \alpha \cdot \dfrac{S_n}{n}, \label{eq:CondExpect}\\
\intertext{and}
	E[S_{n+1} \mid \mathcal{F}_n]&=\left(1+\dfrac{\alpha}{n}\right) S_n. \label{eq:CondExpect_bis}
\end{align}
We set
\begin{align}
a_n:=\prod_{j=1}^{n-1} \left(1+\dfrac{\alpha}{j}\right) = \prod_{j=1}^{n-1} \dfrac{j+\alpha}{j} = \dfrac{\Gamma(n+\alpha)}{\Gamma(n)\Gamma(1+\alpha)}. \label{eq:Defan}
\end{align}
Stirling's formula yields that for $a,b \in \mathbb{R}$,
\begin{align}
		\dfrac{\Gamma(n+a)}{\Gamma(n+b)} \sim n^{a-b} \qquad (n \to \infty), \label{Asymp:RatioOfGammas}
\end{align}
and
\begin{align}
a_n \sim \dfrac{n^{\alpha}}{\Gamma(1+\alpha)} \qquad (n \to \infty),
\label{eq:Asympan}
\end{align}
where $x_n \sim y_n$ means that $x_n/y_n$ converges to $1$ as $n \to \infty$.

Let 
\[ M_n:=\dfrac{S_n}{a_n} \qquad \mbox{for $n \in \mathbb{N}$}. \]
Then $\{M_n\}$ is a martingale with respect to $\{\mathcal{F}_n\}$. 
As $E[M_1] = E[S_1]=\beta$, we have
\begin{align}
E[S_n]=\beta a_n \qquad \mbox{for any $n \in \mathbb{N}$}. 
\label{ElephantRWSnasymp}
\end{align}
By \eqref{eq:Asympan}, the mean displacement is asymptotically of order $n^{\alpha}$, provided that $\beta \neq 0$.

%
%

Sch\"{u}tz and Trimper \cite{SchutzTrimper04} show that there are two distinct (diffusive/super-diffusive) regimes about the asymptotic behavior of the mean square displacement: By a calculation found in page 12 of \cite{Bercu18}, 
\begin{align}
E[(S_n)^2] &= \dfrac{\Gamma(n+2\alpha)}{\Gamma(n)} \sum_{\ell=1}^n  \dfrac{\Gamma(\ell)}{\Gamma(\ell+2\alpha)} \notag \\ 
&=
\begin{cases}
\dfrac{n}{1-2\alpha} + \dfrac{1}{2\alpha-1} \cdot \dfrac{\Gamma(n+2\alpha)}{\Gamma(n)\Gamma(2\alpha)}  &(\alpha \neq 1/2), \\
\displaystyle n\sum_{\ell=1}^n \dfrac{1}{\ell}&(\alpha=1/2).
\end{cases}
\label{eq:E[S_n^2]exact}
\end{align}
Here we regard
\begin{align}
\dfrac{1}{\Gamma(s)} =0\qquad \mbox{for $s=0,-1,-2,\ldots$}.
\label{eq:GammaSingular}
\end{align}
By \eqref{Asymp:RatioOfGammas},
\begin{align}
	E[(S_n)^2] 
\sim 
\begin{cases}
\dfrac{n}{1-2\alpha} &(-1<\alpha<1/2), \\[3mm]
n\log n &(\alpha=1/2), \\[1mm]
\dfrac{n^{2\alpha}}{(2\alpha-1)\Gamma(2\alpha)}  &(1/2<\alpha<1).
\end{cases}
\label{eq:ElephantRWSn2Asymp}
\end{align}
After their study, 
several limit theorems describing the influence of the memory parameter $p$ 
have been studied by many authors  \cite{BaurBertoin16,Bercu18,Collettietal17a,Collettietal17b,KubotaTakei19JSP}. 
%
When $-1<\alpha < 1/2$, the elephant random walk is diffusive, and the fluctuation is Gaussian: 
\begin{align}
 \dfrac{S_n}{\sqrt{n/(1-2\alpha)}} \stackrel{\text{d}}{\to} N(0,1) \quad (n \to \infty).
 \label{eq:ERWCLT<1/2}
\end{align}
where $\stackrel{\text{d}}{\to}$ denotes the convergence in distribution, and $N(0,1)$ is the standard normal distribution. 
When $\alpha = 1/2$, the walk is marginally superdiffusive, but still 
\begin{align}
 \dfrac{S_n}{\sqrt{n \log n}} \stackrel{\text{d}}{\to} N(0,1) \quad (n \to \infty).
 \label{eq:ERWCLT=1/2}
\end{align}
On the other hand, if $\alpha>1/2$, then $\{M_n\}$ is an $L^2$-bounded martingale, and there exists a random variable $M_{\infty}$ such that $P(M_{\infty} \neq 0)>0$ and
\begin{align}
M_n=\dfrac{S_n}{a_n} \to M_{\infty}\quad \mbox{a.s. and in $L^2$ as $n \to \infty$.}
\label{conv:M_nL2}
\end{align}
Although $M_{\infty}$ is non-Gaussian,
\begin{align}
\dfrac{S_n-M_{\infty} \cdot a_n}{\sqrt{n/(2\alpha-1)}} \stackrel{\text{d}}{\to} N(0,1)\quad (n \to \infty).
 \label{eq:ERWCLT>1/2}
\end{align}
%


The rate of convergence in the central limit theorem for $\{S_n\}$ when $-1<\alpha \leq 1/2$ is studied in Fan, Hu, and Ma \cite{FanHuMa21}. 
Let $Z$ be a random variable with the standard normal distribution, and for another random variable $X$, we define 
\[
D(X) := \sup_{x \in \mathbb{R}} |P(X \leq x) - P(Z \leq x)|.
\]
The following Berry--Esseen type bound is a consequence of Corollary 1 in \cite{FanHuMa21} (see Appendix \ref{appendix:FanHuMa21Cor1}):
\begin{equation} \label{eq:FanHuMa21Cor1improved}
\begin{split}
D\left(\dfrac{S_n}{\sqrt{n/(1-2\alpha)}}\right) &\leq 
\begin{cases}
\dfrac{C_{\alpha} \log n}{\sqrt{n}} & (-1<\alpha\leq 0), \\[4mm]
\dfrac{C_{\alpha} \log n}{n^{(1-2\alpha)/2}} & (0\leq \alpha<1/2), \\[4mm]
\end{cases} \\
D\left(\dfrac{S_n}{\sqrt{n\log n}}\right) &\leq \dfrac{C_{1/2} \log \log n}{\sqrt{\log n}}  \quad (\alpha=1/2).
\end{split}
\end{equation}
There is a crossover phenomenon in the {\it upper bounds} at $\alpha=0$. A similar result on Wasserstein-$1$ distance is obtained by Ma, El Machkouri, 
and Fan~\cite{MaElMachkouriFan22}.
Our main results in the next section give the {\it exact} rate of moment convergence in the central moment limit theorem, which describe such transition more clearly.

\section{Main results} \label{sec:MainResults}

We denote by $\mu_k$ the $k$-th moment of the standard normal distribution: 
\[
	\mu_k := \begin{cases}
			0 &(k=2m-1), \\
			(2m-1)!!:=\prod_{j=1}^m (2j-1) &(k=2m),
\end{cases} 
\quad ( m \in \mathbb{N}).
\]
%
%
We show the following moment central limit theorems, which are stronger than  \eqref{eq:ERWCLT<1/2} and \eqref{eq:ERWCLT=1/2} (see Section 30 of \cite{Billingsley12PMAnniversary}, or Section 2.2.3 in \cite{Tao12AMSBook}): If $-1<\alpha<1/2$, then
\begin{align}
 \lim_{n \to \infty} E\left[\left(\dfrac{S_n}{\sqrt{n/(1-2\alpha)}}\right)^k\right] = \mu_k\quad (k \in \mathbb{N}),
\label{eq:ERWmomentCLT<1/2}
\end{align}
and if $\alpha=1/2$, then
\begin{align}
 \lim_{n \to \infty} E\left[\left(\dfrac{S_n}{\sqrt{n \log n}}\right)^k\right] = \mu_k\quad (k \in \mathbb{N}).
\label{eq:ERWmomentCLT=1/2}
\end{align}
Moreover, the {\it exact} rate of convergence in  \eqref{eq:ERWmomentCLT<1/2} and \eqref{eq:ERWmomentCLT=1/2}
are derived.
In particular, as a consequence of our main result, the rate of convergence of the odd order moments and of the second moment becomes 
small as $\alpha$ increases, while there is a transition at $\alpha=0$ about the rate of convergence of the $2m$-th moment with $m \geq 2$.



Our first result is about the odd order moments, which describe the asymmetry of the distribution and reflect the effect of the bias $\beta=2q-1$ of the first step remaining at time $n$. (It follows from (\ref{eq:elephantRWCondDistp}) that if $\beta  =2q-1 = 0$, then $\{S_n\}$ has the same law as the process $\{-S_n\}$, and we have $E[(S_n)^{2m-1}] = 0$ for any $m,n \in \mathbb{N}$.) The next theorem shows that the effect of the bias of the first step is vanishing more rapidly as $\alpha$ decreases.

\begin{theorem} \label{thmHOT:OddMoments} Assume that $\beta=2q-1 \neq 0$. 
		
\noindent (i) If $-1<\alpha<1/2$, then we have, as $n \to \infty$, 
			\begin{align}
 				E\left[\left(\dfrac{S_n}{\sqrt{n/(1-2\alpha)}}\right)^{2m-1}\right] 
					\sim \dfrac{\beta \sqrt{1-2\alpha}}{\Gamma(1+\alpha)} \cdot (2m-1)!! \cdot n^{-(1-2\alpha)/2}.
 				\label{eq:ERWCLTRate<1/2Odd}
			\end{align}
\noindent (ii) If $\alpha=1/2$, then we have, as $n \to \infty$,
			\begin{align}
 				E\left[\left(\dfrac{S_n}{\sqrt{n \log n}}\right)^{2m-1}\right] \sim \dfrac{2\beta}{\sqrt{\pi}}\cdot (2m-1)!! \cdot  (\log n)^{-1/2}.
 			\label{eq:ERWCLTRate=1/2Odd}
			\end{align}
\end{theorem}

%




Theorem~\ref{thmHOT:OddMoments} says that the asymptotics of the higher odd order moments are just a constant multiple of that of the first moment, which is an immediate consequence of (\ref{eq:Asympan}) and (\ref{ElephantRWSnasymp}).  Unlike this, the next theorem shows that for $\alpha<0$, the exponent in the rate of convergence of the $2m$-th moments with $m \geq 2$ is different from that of the second moment, which follows from \eqref{eq:Asympan}, \eqref{eq:GammaSingular}, and \eqref{eq:ElephantRWSn2Asymp}.

\begin{theorem} \label{thmHOT:EvenMoments} The following three assertions hold true: 

\noindent (i) If $-1<\alpha< 0$ and $\alpha \neq -1/2$, then
\begin{equation} \label{eq:ERWCLTRate=otherEven2}
							E\left[\left(\dfrac{S_n}{\sqrt{n/(1-2\alpha)}}\right)^2\right] - 1 \sim -\dfrac{1}{\Gamma(2\alpha)} \cdot n^{-(1-2\alpha)}
							\qquad(n\to \infty). 
						\end{equation}
The left hand side of \eqref{eq:ERWCLTRate=otherEven2} is identically zero if $\alpha=0$ and $n \geq 1$, or if $\alpha=-1/2$ and $n \geq 2$. On the other hand, for any $-1<\alpha \leq 0$ and $m=2,3,\ldots$, we have that, as $n\to \infty$, 		
				\begin{equation}\begin{split}
 					&\frac{1}{(2m-1)!!} E\left[\left(\dfrac{S_n}{\sqrt{n/(1-2\alpha)}}\right)^{2m}\right] - 1 
					 \sim \frac{m(m-1)}{2} \cdot c(\alpha) \cdot 
						 n^{-1},
 							\label{eq:ERWCLTRate<1/2Even}
				\end{split}\end{equation}
			where
			\begin{equation}
			\label{def_of_c}
					c(\alpha)=-\frac{2(2\alpha^2+1)}{3(1-4\alpha) }.
			\end{equation} 
\noindent (ii) If $0<\alpha<1/2$, then we have that, as $n\to\infty$, 
				\begin{align}
					\frac{1}{(2m-1)!!}  E\left[\left(\dfrac{S_n}{\sqrt{n/(1-2\alpha)}}\right)^{2m}\right] - 1
					 	\sim -\dfrac{m  }{\Gamma(2\alpha)} \cdot n^{-(1-2\alpha)}
 							\label{eq:ERWCLTRate(0,1/2)Even}
				\end{align}
				holds for any $m \in \mathbb{N}$.

\noindent (iii) If $\alpha=1/2$, then we have that, as $n\to \infty$,
				\begin{align}
					\frac{1}{ (2m-1)!! } E\left[\left(\dfrac{S_n}{\sqrt{n \log n}}\right)^{2m} \right]- 1  
						 \sim m \gamma \cdot (\log n)^{-1}
 							\label{eq:ERWCLTRate=1/2Even}
				\end{align} 
			holds for any $m \in \mathbb{N}$. Here $\gamma$ is Euler's constant.
\end{theorem}

\begin{remark} \label{rem:PosRecrrence}
Recently, a new critical phenomena about the recurrence property of the elephant random walk has been found inside the diffusive regime. 
Define the first return time
\[ R:=\inf \{ n> 0 : S_n=0 \}. \]
As a consequence of the law of the iterated logarithm (see \cite{Collettietal17b} and \cite{Bercu18}) we can see that $P(R<+\infty)=1$ if $-1<\alpha \leq 1/2$. 
On the other hand, it follows from \eqref{conv:M_nL2} that $P(R=+\infty)>0$ if $\alpha>1/2$. 
 Bertoin \cite{Bertoin22Counting} shows that there is a transition at $\alpha=-1/2$ about the expected first return time: 
\begin{align*}
E[R]\begin{cases}
<+\infty &(-1<\alpha < -1/2), \\
=+\infty &(-1/2\leq \alpha \leq 1/2). 
\end{cases}
\end{align*}
We note that $c(\alpha)$ defined by \eqref{def_of_c} attains its maximum at this critical point $\alpha=-1/2$.
\end{remark}

Our results are summarized in Figure \ref{graph}.

\begin{figure}[hbtp]
 \centering
 \begin{tikzpicture}[scale=0.64]

		%
		%

			\draw[->,>=stealth,semithick] (-10,0)--(-4,0)node[below]{\textrm{\footnotesize $\alpha $}}; 		
			\draw[->,>=stealth,semithick] (-7,-1.0)--(-7,6.5) node[left]{$\gamma$}; 	
			\draw (-6.9,0) node[below left] {\textrm{\scriptsize$O$}};
			\draw (-5,0) node[below]{ \textrm{\footnotesize$1$}};
			\draw[semithick] (-9,3)--(-6,0) node[below]{ \textrm{\footnotesize${\frac{1}{2}}$}}; 
			\draw[thin,dashed] (-8,2)--(-8,0)node[below]{\textrm{\footnotesize$-\frac{1}{2}\,\,$}};
			\draw (2,0)node[below]{\textrm{\footnotesize $1$}};
			\draw[thin,dashed] (-9,3)--(-7,3)node[right]{\textrm{\footnotesize${\frac{3}{2} }$}};
			\draw[thin,dashed] (-9,3)--(-9,0)node[below]{\textrm{\footnotesize${-1}\,\,$}};
			\draw[thin,dashed] (-8,2)--(-7,2)node[right]{\textrm{\footnotesize${ 1}$}};			
			\draw (-7,-1) node[below]{$\mbox{\footnotesize odd order moments}$};

		%
		%

			\draw[->,>=stealth,semithick] (-3,0)--(3,0)node[below]{\textrm{\footnotesize $\alpha $}}; 		
			\draw[->,>=stealth,semithick] (0,-1.0)--(0,6.5) node[left]{$\gamma$}; 
			\draw (0.1,0) node[below left]{\textrm{\scriptsize$O$}};
			\draw (0,2.3) node[right]{\footnotesize$1$};
			\draw[semithick] (-2,2)--(0,2); 
			\draw[semithick] (0,2)--(1,0) node[below]{ \textrm{\footnotesize${\frac{1}{2}}$ }}; 
			\draw (2,0)node[below]{\textrm{\footnotesize${1}$}};
			\draw[thin,dashed] (-2,2)--(-2,0)node[below]{\textrm{\footnotesize${ -1}$}};
			\draw (0,-1) node[below]{$\mbox{\footnotesize$2m$-th order moments ($m\geq2$)}$};

		%
		%

			\draw[->,>=stealth,semithick] (4,0)--(10,0)node[below]{\textrm{\footnotesize$\alpha $}}; 
			\draw[->,>=stealth,semithick] (7,-1)--(7,6.5) node[left]{$\gamma$}; 
			\draw (7.1,0) node[below left] {\textrm{\scriptsize$O$}};
			\draw (7,2.3) node[right]{\footnotesize$1$};
			\draw[semithick] (5,6)--(8,0) node[below]{ \textrm{\footnotesize${\frac{1}{2}}$ }}; 
			\draw (9,0) node[below]{\textrm{\footnotesize $1$}}; 
			\draw[thin,dashed] (5,6)--(5,0)node[below]{\textrm{\footnotesize$-1$}};
			\draw[thin,dashed] (5,6)--(7,6)node[right]{\textrm{\footnotesize${3}$}};
			\draw[thin,dashed] (6,4)--(6,0)node[below]{\textrm{\footnotesize ${- \frac{1}{2} }$}};
			\draw[thin,dashed] (6,4)--(7,4)node[right]{\textrm{\footnotesize${2}$}};
			
			\filldraw[fill=white] (7,2) circle[radius=0.1];
			\filldraw[fill=white] (6,4) circle[radius=0.1];
			\draw (7,-1) node[below]{$\mbox{\footnotesize the second order moment}$};

		\end{tikzpicture}

 \caption{$\alpha$-dependence of the convergence rate. 
%
When $-1<\alpha <1/2$, 
the difference $E[\{S_n/\sqrt{n/(1-2\alpha)}\}^k]-\mu_k$ vanishes like $n^{-\gamma}$. On the other hand, when $\alpha=1/2$, the decay of the difference $E[\{S_n/\sqrt{n\log n}\}^k] -\mu_k$ is much slower than polynomial.
%
We find 
a crossover phenomenon at $\alpha=0$ in the convergence rate of the $2m$-th moment with $m=2,3,\ldots$. For the second moment, there are 
two exceptional points which arise from the singularity of the gamma function.} 
\label{graph}
\end{figure}
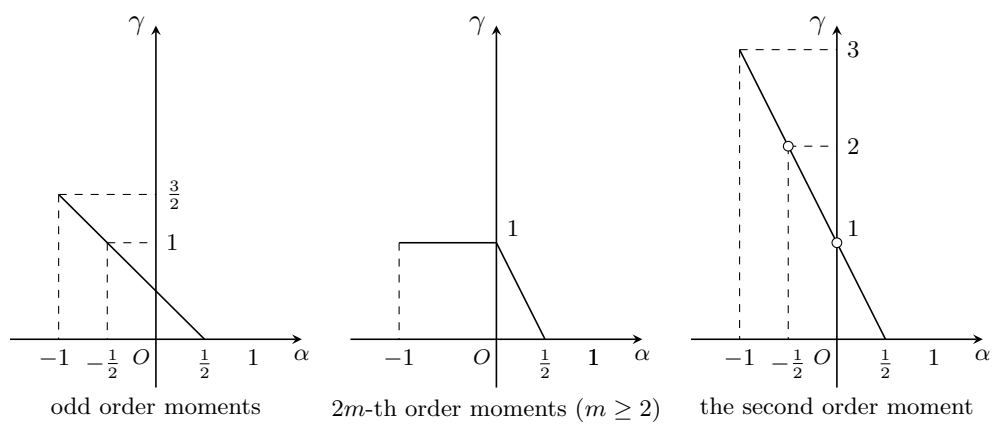

Bercu, Chabanol, and Ruch \cite{BercuChabanolRuch19} derives the asymptotic behavior of the $k$-th moment $E[(S_n)^k]$ of $S_n$ when $\alpha>1/2$. In \cite{BercuChabanolRuch19},
it is shown that $\{M_n\}$ is an $L^k$-bounded martingale for any $k \in \mathbb{N}$, which implies that  
\[ E[(M_n)^k] \to E[(M_{\infty})^k]=:C_k \qquad (n \to \infty). \]
If $k$ is an odd integer and $\beta=0$, then  $C_k=0$.
Except this case, the constant $C_k$ is positive and we have
\[ E[(S_n)^k] \sim C_k (a_n)^k \qquad (n \to \infty). \] 
The main purpose in \cite{BercuChabanolRuch19} is to establish new hypergeometric identities by calculating $E[(M_{\infty})^k]$ in two different ways. One of interesting future problems is to investigate the rate of moment convergence corresponding to \eqref{eq:ERWCLT>1/2}.

\section{Proof of Theorem~\ref{thmHOT:OddMoments}}\label{sec:thmHOT:OddMoments}

%
%
%
%
%
%
%
%

Let $m,n \in \mathbb{N}$. 
We start with the following recursion, which is a consequence of (4.3) in \cite{BercuChabanolRuch19}: 
	\begin{align}
			E[(S_{n+1})^{2m-1}]
				&=\sum_{\ell =1}^m\left\{\binom{2m -1}{2\ell-1}+\frac{\alpha}{n}\binom{2 m-1}{2 \ell -2}\right\} E[(S_{n})^{2\ell -1}].
		\label{eq:Oshiro-odd}
	\end{align}
We set 
\begin{equation}\label{def_of_f_odd}
	\begin{split}
	M^{(2m-1)}_n &:= E[(S_n)^{2m-1}],\\
	f^{(2m-1)}_n 
		&:={\displaystyle \sum_{\ell =1}^{m-1}\left\{\binom{2m -1}{2\ell-1}
		+\frac{\alpha}{n}\binom{2 m-1}{2 \ell -2}\right\} M^{(2\ell-1)}_n}, \\ 
	g^{(2m-1)}_n&:= 1+\frac{(2m-1)\alpha}{n}. 
	\end{split}
\end{equation}
Then the recursion \eqref{eq:Oshiro-odd} can be put into the 
following form: For $n \in \mathbb{N}$,
\begin{equation}\label{eq:Oshiro}
	M_{n+1}^{(2m-1)} =f_n^{(2m-1)}+ g^{(2m-1)}_n	M_n^{(2m-1)}.
\end{equation}
It is convenient to define
\[ f_0^{(2m-1)}:=M_1^{(2m-1)}=E[(X_1)^{2m-1}]=\beta\quad \mbox{and}\quad g_0^{(2m-1)} := 0. \]
To solve \eqref{eq:Oshiro}, we use the following two lemmas. ($j_0$ is introduced in order to avoid issues associated with singularities \eqref{eq:GammaSingular} of the gamma function.)

\begin{lemma}[cf. Sch\"{u}tz and Trimper \cite{SchutzTrimper04}, (12) and (13)] \label{SchutzTrimper04Recursion} Let $j_0 \in \mathbb{N}$. 
The general solution $\{x_n\}_{n={ j_0}+1,{\ j_0}+2,\ldots}$ to the recursion 
\[ \begin{cases} x_{ j_0+1}=f_{ j_0}, \\
x_{n+1} = f_n+g_n \cdot x_n &(n={ j_0}+1,{ j_0}+2,\ldots) \\
\end{cases} \]
is given by
\[ x_n = { x_{j_0+1} \prod_{k=j_0+1}^{n-1} g_k +} \sum_{j=j_0 +1}^{n-1} f_j \cdot \prod_{k=j+1}^{n-1} g_k\qquad (n= j_0+1, j_0+2,\ldots). \]
\end{lemma}

\begin{lemma} \label{lem:product210826} For $\delta \in \mathbb{R}$, we define %
\begin{align}\label{def_of_j}
			 j_0(\delta) 
			&:= \begin{cases}
				k  &\mbox{if $\delta =-k$ for some $k\in \mathbb{N}$,}  \\
 				0 &\mbox{otherwise},\\
			\end{cases}
\end{align}
and
	\begin{align*}
			A_{\delta} &:= 
			\begin{cases}
				k! &\mbox{if $\delta =-k$ for some $k\in \mathbb{N}$,} \\[2mm]
				\dfrac{1}{\Gamma(1+\delta)} & \mbox{otherwise.} \\
			\end{cases}
	\end{align*}
	Then we have
	\begin{align*}
		\prod_{j=j_0(\delta)+1}^{n-1} \left(1+\dfrac{\delta}{j}\right) \sim A_{\delta} \cdot n^{\delta} \qquad (n \to \infty).
	\end{align*}
\end{lemma}

\begin{proof}
	If $\delta$ is not a negative integer, then for any $n>1$,
		\begin{align}
			\prod_{j=  j_0 (\delta)+1}^{n-1} \left(1+\dfrac{\delta}{j}\right) = \prod_{j=1}^{n-1} \dfrac{j+\delta}{j} = \dfrac{\Gamma(n+\delta)}{\Gamma(n)\Gamma(1+\delta)}.
			 \label{eq:product210826_A} 
		\end{align}
If $\delta=-k$ for some $k \in \mathbb{N}$, then for any $n>k+1$,
		\begin{align}
			\prod_{j= j_0(\delta)+1}^{n-1} \left(1+\dfrac{\delta}{j}\right) 
				= \prod_{j=k+1}^{n-1} \dfrac{j-k}{j} 
				 =\dfrac{\Gamma(n-k) \times k!}{\Gamma(n)}.
				\label{eq:product210826_B} 
		\end{align}
We obtain the desired result by \eqref{Asymp:RatioOfGammas}. 
\end{proof}

We set $j_0=j_0((2m-1)\alpha)$. Note that
\[ g_{ j_0}^{(2m-1)} =0, \quad \mbox{and} \quad 
 g_n^{(2m-1)} \neq 0 \quad \mbox{ for all $n>j_0$.}  \]
By Lemma \ref{SchutzTrimper04Recursion}, $\{ M_n^{(2m-1)} \}$ satisfies 
	\begin{equation}\label{solution}\begin{split}
			M^{(2m-1)}_n 
			&= M^{(2m-1)}_{ j_0+1}  %
						{\bar g}^{(2m-1)}_n
				+	\bar{g}_n^{(2m-1)}	\sum_{j= j_0+1}^{n-1} \frac{ f_j^{(2m-1)}}{ \bar{g}_{j+1}^{(2m-1)}},		
	\end{split}\end{equation}
where
	\begin{align}\label{def_of_a}
		\bar{g}_n^{(2m-1)} :=\prod_{k= j_0+1}^{n-1} g^{(2m-1)}_k\qquad (n= j_0+1, j_0+2,\ldots).
	\end{align}
From  Lemma \ref{lem:product210826}, we have
\begin{align}
		\bar{g}_n^{(2m-1)} \sim A_{(2m-1)\alpha} n^{(2m-1)\alpha}\qquad (n \to \infty).
		\label{Asymp:gBarn2m-1}
\end{align}
We put
\begin{align}
\label{def:DefF2m-1nOdd}
F^{(2m-1)}_n := \bar{g}_n^{(2m-1)}	\sum_{j= j_0+1}^{n-1} \frac{ f_j^{(2m-1)}}{ \bar{g}_{j+1}^{(2m-1)}}.
\end{align}

\subsection{Proof of \eqref{eq:ERWCLTRate<1/2Odd} in Theorem~\ref{thmHOT:OddMoments}} Assume that $-1<\alpha<1/2$ and $\beta \neq 0$. We define
\[
		C_{\ell} =C_{\ell}(\alpha, \beta) := \dfrac{(2\ell-1)!!}{(1-2\alpha)^{\ell-1}} \cdot { \dfrac{\beta}{\Gamma(1+\alpha) }} \qquad (\ell \in \mathbb{N}). 
\]
It suffices to show that
\begin{align}
		 M_n^{(2\ell-1)} &\sim  C_{\ell} n^{\ell-1+\alpha} \qquad (n\to \infty)
			\label{Eeq:ERWCLTRate<1/2Odd'}
\end{align}
for $\ell \in \mathbb{N}$, by induction.  
%
%
We note that \eqref{Eeq:ERWCLTRate<1/2Odd'} with $\ell=1$ follows from \eqref{eq:Asympan} and \eqref{ElephantRWSnasymp}. Assume that $m>1$ and \eqref{Eeq:ERWCLTRate<1/2Odd'} holds true for $\ell=1,\ldots,m-1$.
Then we have that, as $n\to \infty$,
	\begin{equation*}
		f_n^{(2m-1)}\sim \binom{2m-1}{2m-3} M_n^{(2m-3)}\sim (2m-1)(m-1)C_{m-1} n^{m-2+\alpha}.
	\end{equation*}
	%
By \eqref{Asymp:gBarn2m-1} and Lemma~\ref{lem:StolzCesaro} in Appendix \ref{app:calculus}, we have
	\begin{align} 
		F^{(2m-1)}_n
		& \sim  A_{(2m-1)\alpha}n^{(2m-1)\alpha} \sum_{j=j_0+1}^{n-1} \dfrac{(2m-1)(m-1) C_{m-1}j^{m-2+\alpha}}{A_{(2m-1)\alpha}(j+1)^{(2m-1)\alpha}} \notag\\
		&\sim (2m-1)(m-1) C_{m-1}    \cdot n^{(2m-1)\alpha} \sum_{j=j_0+1}^{n-1} j^{m-2+\alpha-(2m-1)\alpha} \notag \\
		&\sim  \frac{ (2m-1)(m-1)}{m-1+\alpha-(2m-1)\alpha} \cdot C_{m-1} n^{m-1+\alpha} \notag \\
		&= \frac{2m-1}{1-2\alpha} \cdot C_{m-1}n^{m-1+\alpha} =C_mn^{m-1+\alpha},\label{asymp:oddSubcriticalFn}
	\end{align}
as $n \to \infty$. Since $m>1$ and $-1<\alpha<1/2$, we can see that $m-1+\alpha >(2m-1)\alpha$. 
From \eqref{solution}, \eqref{Asymp:gBarn2m-1} and \eqref{asymp:oddSubcriticalFn},
	\[
		M_n^{(2m-1)} \sim F^{(2m-1)}_n \sim
		C_mn^{m-1+\alpha}
		\qquad (n \to \infty).
	\]
This completes the proof. \qed

%
%
%
%

\subsection{Proof of \eqref{eq:ERWCLTRate=1/2Odd} in Theorem~\ref{thmHOT:OddMoments}}
Here we assume that $\alpha=1/2$ and $\beta \neq 0$. 
It suffices to show that for $\ell \in \mathbb{N}$,
\begin{align}
		M^{(2\ell-1)}_n& \sim C'_\ell \cdot n^{\ell-1/2} (\log n)^{\ell-1} \qquad (n \to \infty)
	\label{eq:ERWCLTRate=1/2Odd'}
\end{align}
holds true. Here constants $C'_\ell$ are given by 
\[ 
	C'_\ell = C'_\ell(\beta):= \dfrac{2\beta}{\sqrt{\pi}} \cdot (2\ell-1)!!. 
\]
%
We  can see that \eqref{eq:Asympan} and \eqref{ElephantRWSnasymp} imply \eqref{eq:ERWCLTRate=1/2Odd'} with $\ell=1$. Assume that $m>1$ and \eqref{eq:ERWCLTRate=1/2Odd'} holds true for  $\ell=1,\ldots,m-1$.
Then we have that, as $n\to \infty$,
\begin{align*}
	f_n^{(2m-1)} &\sim \binom{2m-1}{2m-3} M^{(2m-3)}_n \\
			&\sim (2m-1)(m-1) C'_{m-1}  n^{m-3/2} (\log n)^{m-2}. 
\end{align*}
By \eqref{Asymp:gBarn2m-1} and Lemma~\ref{lem:StolzCesaro}, we have, as $n\to \infty$, 
\begin{align} 
F^{(2m-1)}_n & \sim A_{m-1/2}n^{m-1/2} \sum_{j=1}^{n-1} \dfrac{(2m-1)(m-1) C'_{m-1}j^{m-3/2}(\log j)^{m-2}}{A_{m-1/2}(j+1)^{m-1/2}}  \notag\\
&\sim  (2m-1)(m-1)C'_{m-1} n^{m-1/2} \sum_{j=1}^{n-1} \dfrac{(\log j)^{m-2}}{j} \notag\\
&\sim (2m-1) C'_{m-1} n^{m-1/2} (\log n)^{m-1} \notag \\
&=C'_mn^{m-1/2} (\log n)^{m-1}. \label{asymp:oddCriticalFn} 
\end{align}
Here we used \eqref{eq:lemma-asymptotic2} in Appendix \ref{app:calculus}.
As $m>1$, it follows from \eqref{solution},  \eqref{Asymp:gBarn2m-1} and \eqref{asymp:oddCriticalFn} that
\[
	M_n^{(2m-1)}  \sim F^{(2m-1)}_n \sim C'_m n^{m-1/2} (\log n)^{m-1}\qquad (n\to \infty).
\]
This completes the proof.\qed

\section{Proof of Theorem~\ref{thmHOT:EvenMoments} {\rm (i)} and {\rm (ii)}}\label{sec:thmHOT:EvenMoments}

As we explained in the previous section, the odd order moments of $S_n$ satisfy the non-autonomous system of first order linear difference equations \eqref{eq:Oshiro-odd}. The even order moments of $S_n$ also satisfy the following recursion, 
which can be derived from (4.2) in ~\cite{BercuChabanolRuch19}: For $m,n \in \mathbb{N}$.
\begin{align}
				E[(S_{n+1})^{2m}]
				&=1+\sum_{\ell =1}^m \left\{\binom{2 m}{2 \ell }+\frac{\alpha}{n}\binom{2 m}{2 \ell -1}\right\} E[(S_{n})^{2\ell }].
		\label{eq:Oshiro-even}
	\end{align}
In this section we assume that $-1<\alpha<1/2$, and put
\[ 
	M_n^{(2m)} :=\frac{1}{(2m-1)!!} \,
							E\left[\left(\dfrac{S_n}{\sqrt{n/(1-2\alpha)}}\right)^{2m}\right]   
						  - 1 \qquad (n, m\in \mathbb{N}).
\]
(To treat the case $\alpha=1/2$ we need to another special normalization of the even moments, and the proof for that case is postponed to the next section.) As we shall see below, $M_n^{(2m)}$ can be considered as the solution to a non-autonomous and {\it non-homogeneous} system of first order linear difference equations. In particular, the crossover phenomenon described in Theorem~\ref{thmHOT:EvenMoments} (i) and (ii) arises out of this non-homogeneity.

The asymptotic behavior of $M_n^{(2)}$ can be obtained from  \eqref{Asymp:RatioOfGammas} and \eqref{eq:E[S_n^2]exact}. For the sake of completeness, we give a short proof for $-1<\alpha<1/2$ ($\alpha=1/2$ will be treated in Section~\ref{sec:ThmEven(iii)} below). 
Using \eqref{eq:Oshiro-even}, one can show that  $\{M^{(2)}_n\}$ satisfies the following recursion:
\[
	M^{(2)}_1=-2\alpha,\qquad M^{(2)}_{n+1} =\frac{n}{n+1} \left(1+\frac{2\alpha}{n}\right) M_n^{(2)}\quad (n \in \mathbb{N}). 
\]
If $-1<\alpha<1/2$ and $\alpha \neq 0,-1/2$, then by \eqref{Asymp:RatioOfGammas},
\begin{align}
	M^{(2)}_n =\frac{M^{(2)}_1}{n}\prod^{n-1}_{j=1}\left(1+\frac{2\alpha}{j}\right)=-\frac{\Gamma(n+2\alpha)}{\Gamma(n+1)\Gamma(2\alpha)}
		\sim -\frac{n^{-(1-2\alpha) }}{\Gamma(2\alpha)}, \label{Asymp:Mn2Main}
\end{align}
as $n\to \infty$, while if $\alpha=0$ and $n \geq 1$, or if $\alpha=-1/2$ and $n \geq 2$, then
\begin{align}
M^{(2)}_n=0. 	
\label{Asymp:Mn2Degenerate}
\end{align}

Putting  
\[
	s^{(2m)}_n =\left(\frac{ n}{1-2\alpha}\right)^m (2m-1)!!, 
\]
we have $E[(S_n)^{2m}]=s^{(2m)}_n (M_n^{(2m)}+1) $. 
Hence \eqref{eq:Oshiro-even} can be put into the following form:
\[\begin{split}
	M^{(2m)}_{n+1} 
	&=\frac{1}{s^{(2m)}_{n+1}}
	\left(	1+\sum^m_{\ell=1}
		 \left\{ 	
		 	\binom{2m}{2\ell}	+	\frac{\alpha}{n}  \binom{2m}{2\ell-1}	
		\right\} 	s^{(2\ell)}_n
					M_n^{(2\ell)}\right.\\
		&\qquad +\left. 
		\sum^m_{\ell=1}
		 \left\{ 	
		 	\binom{2m}{2\ell}	+	\frac{\alpha}{n}  \binom{2m}{2\ell-1}	
		\right\}s^{(2\ell)}_n	
		  \right)-1.
\end{split}
\]
Rearranging the right-hand side,
\[
\begin{split}
M^{(2m)}_{n+1} &=\left(\frac{n}{n+1}\right)^m\left( 1+\frac{2m\alpha}{n}\right) M_n^{(2m)}\\
		&\qquad +\frac{1}{s^{(2m)}_{n+1}}\, \sum^{m-1}_{\ell=1}
		 \left\{ 	
		 	\binom{2m}{2\ell}	+	\frac{\alpha}{n}  \binom{2m}{2\ell-1}	
		\right\} 	s^{(2\ell)}_n
					M_n^{(2\ell)}\\
		&\qquad 	
			+\frac{1}{s^{(2m)}_{n+1}} \,\left\{ 1+\sum^m_{\ell=1}
		 \left\{ 	
		 	\binom{2m}{2\ell}	+	\frac{\alpha}{n}  \binom{2m}{2\ell-1}	
		\right\}s^{(2\ell)}_n	\right\}-1.	
\end{split}\]
Hence, by putting
\begin{align}\label{def_of_gfh}\begin{split}	
&f^{(2m)}_n	:= 
			\frac{1}{s^{(2m)}_{n+1} }
			 \sum^{m-1}_{\ell=1}
		 \left\{ 	
		 	\binom{2m}{2\ell}	+	\frac{\alpha}{n}  \binom{2m}{2\ell-1}	
		\right\}
			s^{(2\ell)}_nM^{(2\ell)}_n, \\[2mm]
	&g^{(2m)}_n	:= \left(\frac{n}{n+1}\right)^m \left(1+\frac{2m\alpha}{n}\right),\\[2mm]
	&h^{(2m)}_n 
		:=  \frac{1}{s^{(2m)}_{n+1}} \left\{1+ \,\sum^m_{\ell=1}
		 \left\{ 	
		 	\binom{2m}{2\ell}	+	\frac{\alpha}{n}  \binom{2m}{2\ell-1}	
		\right\}s^{(2\ell)}_n\right\}	-1,
\end{split}\end{align}
we have 
\begin{equation}\label{eq:Even}
	M_{n+1}^{(2m)} = f_n^{(2m)} +h_n^{(2m)} + g_n^{(2m)} M_n^{(2m)}.
\end{equation}
By Lemma \ref{SchutzTrimper04Recursion}, the recursion \eqref{eq:Even} is solved by 
\begin{equation}\label{representation_M} 
	M_n^{(2m)} =  M_{j_0+1}^{(2m)} \,\bar{g}^{(2m)}_n  
			+F^{(2m)}_n+ H^{(2m)}_n.
\end{equation}
Here $j_0=j_0(2m\alpha)$ is defined by \eqref{def_of_j}, and we put
\begin{align}\label{def_of_g_bar}
	\bar{g}_n^{(2m)} &:=\prod_{j=j_0+1}^{n-1} g^{(2m)}_j	
		=\left(\frac{j_0+1}{n}\right)^m \prod^{n-1}_{j=j_0+1} \left(1+\frac{2m\alpha}{j}\right),
\end{align}	
and 
\begin{align}
\begin{split}
	F^{(2m)}_n&:=\bar{g}_n^{(2m)}  \sum_{j=j_0+1}^{n-1} \dfrac{f_j^{(2m)}}{\bar{g}_{j+1}^{(2m)} },\qquad
	H_n^{(2m)} :=\bar{g}_n^{(2m)}  \sum_{j=j_0+1}^{n-1} \dfrac{h_j^{(2m)}}{\bar{g}_{j+1}^{(2m)} }.
\end{split}\end{align}
Applying Lemma~\ref{lem:product210826}, we have 
\begin{align}\label{eq5-4}
		\bar{g}_n^{(2m)} 	
		\sim (j_0+1)^mA_{2m\alpha} n^{-m(1-2\alpha)}
\end{align}
as $n\to \infty$.

Next, let us investigate the asymptotic behavior of $H^{(2m)}_n$, which does not depend on any even order moments of $S_n$.


\begin{lemma}\label{lemma_H} 
 Assume that $m\geq 2$. \\
(i) If $-1<\alpha \leq 0$, then 
\begin{equation}
	H^{(2m)}_n	\sim \frac{(1-4\alpha)c_{2m} }{m(1-2\alpha)-1} \cdot n^{-1}\qquad (n \to \infty).
	\label{asymp:H2m_nOriginal}
\end{equation}
Here $c_{2m}=\dfrac{m(m-1)}{2} c(\alpha)$, and $ c(\alpha)$ is the constant defined by \eqref{def_of_c}. \\
(ii) If $0<\alpha<1/2$, then $H^{(2m)}_n = o (n^{-(1-2\alpha)})$ as $n \to \infty$.

\end{lemma}

\begin{proof} 
First we show 
\begin{equation}
	h_n^{(2m)} 
		\sim -\frac{m(m-1) }{3} \cdot (2\alpha^2+1)  \cdot n^{-2}\quad (n \to \infty).
\label{eq:lemma_h}
\end{equation}
Noting that
\begin{align*}
\frac{s^{(2\ell)}_n}{s^{(2m)}_{n+1}} &= \dfrac{(2\ell-1)!!}{(2m-1)!!} \cdot (1-2\alpha)^{m-\ell} \cdot \dfrac{n^{\ell}}{(n+1)^m} \\
&\sim \dfrac{(2\ell-1)!!}{(2m-1)!!} \cdot (1-2\alpha)^{m-\ell} \cdot n^{\ell-m} \quad (n \to \infty),
\end{align*}
the sum of the terms in $h_n^{(2m)}$ which apparently larger than $O(n^{-2})$ is
\[\begin{split}
	I_n &:=\left\{ \binom{2m}{2m} + \dfrac{\alpha}{n} \binom{2m}{2m-1} \right\} \dfrac{s^{(2m)}_n}{s^{(2m)}_{n+1}} 
			+\binom{2m}{2m-2}  \dfrac{s^{(2m-2)}_n}{s^{(2m)}_{n+1}}-1 \\
	  &=\left(1+\dfrac{2m\alpha}{n}\right) \cdot \left( \dfrac{n}{n+1} \right)^m   + m (1-2\alpha) \cdot \dfrac{n^{m-1}}{(n+1)^m} -1\\
		&= - \dfrac{(n+1)^m - n^m-mn^{m-1}}{(n+1)^m} \sim -\frac{m(m-1)}{2} \cdot n^{-2} \quad (n \to \infty).
\end{split}\]
The $O(n^{-2})$ terms remaining in $h_n^{(2m)}$ are
\[\begin{split}
	J_n &:= \dfrac{\alpha}{n}\binom{2m}{2m-3}  \dfrac{s^{(2m-2)}_n}{s^{(2m)}_{n+1}} +\binom{2m}{2m-4} \frac{s_n^{(2m-4)} }{s^{(2m)}_{n+1}} \\
	&\sim \frac{2m(m-1)}{3} \cdot \alpha (1-2\alpha) \cdot n^{-2} + \frac{m(m-1)}{6} \cdot ( 1-2\alpha)^2 \cdot n^{-2} \\
	&= \dfrac{m(m-1)}{6} \cdot (1-4\alpha^2) n^{-2}\quad (n \to \infty).
\end{split}
\]
Since $h_n^{(2m)}-I_n-J_n = o(n^{-2})$ as $n \to \infty$, we obtain \eqref{eq:lemma_h}.

If $-1<\alpha \leq 0$ and $m \geq 2$, then $-2+m(1-2\alpha) \geq 0$.
Using \eqref{eq5-4}, \eqref{eq:lemma_h}, 
and Lemma~\ref{lem:StolzCesaro}, we have that, as $n\to \infty$, 
\begin{equation*}\begin{split}
		H^{(2m)}_n 	&\sim  
			-\frac{ m(m-1) }{3}  \cdot (2\alpha^2+1) \cdot n^{-m(1-2\alpha)  }  
			\sum^{n-1}_{j=1} j^{-2+m(1-2\alpha)}\\
			&
			\sim    -\frac{ m(m-1) }{3} \cdot \frac{2\alpha^2+1}{m(1-2\alpha)-1}\cdot n^{-1}
			 =
			 \frac{ (1-4\alpha)c_{2m} }{m(1-2\alpha)-1} \cdot n^{-1}. 
\end{split}\end{equation*}
This proves (i). If $0<\alpha < 1/2$ and $m \geq 2$, then as $n \to \infty$,
\begin{align*}
H^{(2m)}_n = \begin{cases}
O(n^{-m(1-2\alpha)  }) &\mbox{if $-2+m(1-2\alpha)<-1$}, \\
O(n^{-m(1-2\alpha)  }\log n)  &\mbox{if $-2+m(1-2\alpha)=-1$}, \\
O(n^{-1})  &\mbox{if $-2+m(1-2\alpha)>-1$}, 
\end{cases}
\end{align*}
which implies (ii).
\end{proof}


\subsection{Proof of \eqref{eq:ERWCLTRate(0,1/2)Even} in Theorem~\ref{thmHOT:EvenMoments}}\label{subsec:EvenMoments(0,1/2)}

Here we suppose that $0<\alpha<1/2$. 
By induction with respect to $m$, we will prove \eqref{eq:ERWCLTRate(0,1/2)Even}. 
The case $m=1$ is proved in \eqref{Asymp:Mn2Main}.
Let $m\geq 2$ be an integer. For $\ell=1,\ldots,m-1$, suppose that
\[
	M_n^{(2\ell)} \sim -\frac{\ell}{\Gamma(2\alpha)}   n^{-(1-2\alpha)}	\qquad (n \to \infty)
\]
holds true. 
By the hypothesis of the induction, we have
\[\begin{split}
	f^{(2m)}_n 
		\sim \binom{2m}{2m-2} \frac{s_n^{(2m-2)}}{s_{n+1}^{(2m)}}M_n^{(2m-2)}
		\sim -(1-2\alpha) \frac{m(m-1)}{\Gamma(2\alpha)}  \cdot n^{-2(1-\alpha)}  
\end{split}\]
as $n\to \infty$.
By \eqref{eq5-4} and Lemma~\ref{lem:StolzCesaro}, we have
\begin{equation}
\begin{split}
	F_n^{(2m)}&\sim  -(1-2\alpha) \frac{m(m-1)}{\Gamma(2\alpha)}
	n^{-m(1-2\alpha)} \sum^{n-1}_{j=1} j^{-2(1-\alpha)+m(1-2\alpha)}\\
	&\sim  -\frac{m}{\Gamma(2\alpha)} n^{-(1-2\alpha)}
\end{split}
\label{Asymp:Fn2malpha>0}
\end{equation}
as $n\to \infty$. 
Here we used the fact that
	\[
		-2(1-\alpha)+m(1-2\alpha)\geq -2(1-\alpha)+2(1-2\alpha)=-2\alpha>-1.
	\]	
In view of \eqref{representation_M}, \eqref{eq5-4} and Lemma~\ref{lemma_H} (ii), we have $M^{(2m)}_n \sim F_n^{(2m)}$ as $n\to\infty$.
This completes the proof. \qed


\subsection{Proof of \eqref{eq:ERWCLTRate<1/2Even} in Theorem~\ref{thmHOT:EvenMoments}}
Suppose that $-1<\alpha \leq  0$. 
By induction with respect to $m\geq 2 $, we will prove \eqref{eq:ERWCLTRate<1/2Even}. 
By \eqref{def_of_gfh}, 
\[
f^{(4)}_n	= 
			\frac{s^{(2)}_n}{s^{(4)}_{n+1} }
			\left\{ 	
		 	\binom{4}{2}	+	\frac{\alpha}{n}  \binom{4}{1}	
		\right\}M^{(2)}_n.
\]
By \eqref{Asymp:Mn2Degenerate}, 
\[
f^{(4)}_n	= F^{(4)}_n = 0 \qquad \mbox{for $n \geq 2$, if $\alpha=0$ or $-1/2$.}
\]
If $-1<\alpha<0$ and $\alpha \neq -1/2$, then \eqref{Asymp:Mn2Main} and the same  computation leading \eqref{Asymp:Fn2malpha>0} yield that
\[
F^{(4)}_n \sim -  \dfrac{2}{ \Gamma(2\alpha) }  n^{ -(1-2\alpha)  } \qquad (n \to \infty).
\]
Taking this, \eqref{representation_M}, \eqref{eq5-4} and Lemma~\ref{lemma_H} (i) into consideration, we have that, as $n\to \infty$,
\[
M^{(4)}_n \sim H_n^{(4)} \sim \dfrac{(1-4\alpha)c_4}{2(1-2\alpha)-1} \cdot n^{-1} = c_4 n^{-1}.
\]
Therefore \eqref{eq:ERWCLTRate<1/2Even} with $m=2$ holds true.

Now, let $m\geq 3$ be an integer, and assume that for $\ell=2,\ldots,m-1$,
\[
	M_n^{(2\ell)} \sim c_{2\ell}  \,n^{-1}\qquad (n \to \infty)
\]
holds true. 
%
The hypothesis of the induction yields that, as $n\to\infty$, 
\[\begin{split}
	f^{(2m)}_n &\sim 
				\binom{2m}{2m-2} \frac{s_n^{(2m-2)}}{s_{n+1}^{(2m)}}M_n^{(2m-2)}
				\sim (1-2\alpha) m  \, c_{2m-2} \,n^{-2}.
\end{split}\]
Using \eqref{eq5-4} and Lemma~\ref{lem:StolzCesaro} again, we have that, as $n\to \infty$, 
\begin{equation*}\begin{split}
		F^{(2m)}_n 
		 	&\sim     (1-2\alpha) m   c_{2m-2}  \,n^{-m(1-2\alpha)  } \sum^{n-1}_{j=j_0+1} j^{-2+m(1-2\alpha)}\\
			&\sim   \frac{     (1-2\alpha)  m c_{2m-2}     		}{m(1-2\alpha)-1} \cdot n^{-1}
		=\frac{(1-2\alpha)(m-2)}{m(1-2\alpha)-1} \cdot c_{2m} \, n^{-1}.
\end{split}\end{equation*}
By \eqref{representation_M}, \eqref{eq5-4} and Lemma~\ref{lemma_H} (ii),
we have that, as $n\to \infty$, 
\[\begin{split}
	M^{(2m)}_n 
	&\sim H^{(2m)}_n+F^{(2m)}_n\\
	&\sim \left\{  \frac{1-4\alpha}{m(1-2\alpha)-1}+ \frac{(1-2\alpha)(m-2)}{m(1-2\alpha)-1}    \right\}  c_{2m} \,n^{-1} =c_{2m} \,n^{-1}.
\end{split}\]
Thus we obtain the desired result. 
\qed

%
%
%
%
%
%
%
%

\section{Proof of Theorem~\ref{thmHOT:EvenMoments} {\rm (iii)}}
\label{sec:ThmEven(iii)}

Here we suppose that $\alpha=1/2$. 
Then \eqref{eq:Oshiro-even} can be put into the 
following form:
	\begin{equation}\begin{split}
		E[(S_{n+1})^{2m}] 
		&=1+ \sum^m_{\ell=1} \left\{ \binom{2m}{2\ell} +\frac{1}{2n}\binom{2m}{2\ell-1} \right\}E[(S_n)^{2\ell}].			
	\end{split} \label{eq6-1} \end{equation}
We introduce an auxiliary sequence: Define
	 \[	 
			L^{(2m)}_n:=E\left[ (S_n)^{2m} \right]  \left/ \binom{n+m-1}{n-1}\right. 
			=\frac{(n-1)!\,m!}{(n+m-1)!}  E\left[ (S_n)^{2m} \right].
	\]
As can be seen by \eqref{eq6-1}, $\{L^{(2m)}_n\}_{n=1,2,\ldots}$ satisfies
	\[\begin{split}
		L^{(2m)}_{n+1}-L^{(2m)}_n
			&=    \binom{n+m}{n}^{-1}   \left(1+ \sum^{m-1}_{\ell=1} \left\{ \binom{2m}{2\ell} +\frac{1}{2n}\binom{2m}{2\ell-1} \right\}E[(S_n)^{2\ell}]\right),
	\end{split}\]
and we obtain
	\begin{equation}\begin{split}
		L^{(2m)}_n 
		&=		1+{ \sum^{n-1}_{j=1}  \binom{j+m}{j}^{-1} }\\
		&
\qquad +	\sum^{n-1}_{j=1} \sum^{m-1}_{\ell=1} \left\{ \binom{2m}{2\ell} +\frac{1}{2j}\binom{2m}{2\ell-1} \right\} \binom{j+m}{j}^{-1}  
					E[(S_j)^{2\ell}].
	\end{split}\label{eq7-2}\end{equation}	
Here we have used the fact  that $L_1^{(2m)} = 1$. Now, we put
\[\begin{split}
	M_n^{(2m)} &:= \frac{1}{(2m-1)!!} E\left[\left(\frac{S_n}{\sqrt{n \log n }}\right)^{2m}\right]-1,\\[2mm]
	t_n^{(2m)} &:=	(n\log n)^m	\cdot (2m-1)!!\left/\binom{n+m-1}{n-1}  \right. .
\end{split}\] 
Then we have $L_n^{(2m)}=t_n^{(2m)}(M_n^{(2m)}+1)$. Hence, \eqref{eq7-2} can be put into the following form:
\begin{equation}
	M_n^{(2m)} =I^{(2m)}_n +J^{(2m)}_n+K^{(2m)}_n,
\end{equation}
where
\begin{equation}\label{eq7-3}\begin{split}
	I^{(2m)}_n :=&\,\frac{1}{t_n^{(2m)}}   \left(1+\sum^{n-1}_{j=1} \binom{j+m}{j}^{-1}	\right),  \\[2mm]
	J^{(2m)}_n :=&
		 \frac{ 1 }{t_n^{(2m)}}   
		 		\sum^{n-1}_{j=1} \sum^{m-1}_{\ell=1} \left\{ \binom{2m}{2\ell} +\frac{1}{2j}\binom{2m}{2\ell-1} \right\}  
		d_j^{(2\ell)}   -1,\\[2mm]
	K^{(2m)}_n :=&  \frac{   1 }{t_n^{(2m)}}  \sum^{n-1}_{j=1} \sum^{m-1}_{\ell=1} \left\{ \binom{2m}{2\ell} +\frac{1}{2j}\binom{2m}{2\ell-1} \right\} 
					d_j^{(2\ell)} M_j^{(2\ell)},
\end{split}
\end{equation}
and
\[
d_j^{(2\ell)}:=\binom{j+m}{j}^{-1} \binom{j+\ell-1}{j-1}  t_j^{(2\ell)}.
\]

We first investigate the asymptotic behaviors of $I^{(2m)}_n$ and $J^{(2m)}_n$, which do not depend on any even order moments of $S_n$. As in the case $0<\alpha<1/2$ in Section \ref{subsec:EvenMoments(0,1/2)}, we will see that they are negligible compared with $K^{(2m)}_n$.

\begin{lemma}\label{lem5-2}
	For any $m\geq 2$, we have that, as $n\to \infty$, 
		\begin{align}
			I^{(2m)}_n&=O((\log n)^{-m}) , \label{eq5-5}\\
			J^{(2m)}_n&=O((\log n)^{-m}) . \label{eq5-6} 
		\end{align}
\end{lemma}
\begin{proof}
	To prove \eqref{eq5-5}, note that
	\begin{align}
	\binom{j+m}{j} &\sim \dfrac{j^m}{m!} \qquad (j \to \infty),\label{asymp:binomj+mj}\\
		t_n^{(2m)} &\sim m! \cdot (2m-1)!! \cdot  (\log n)^{m} \qquad (n \to \infty). \label{eq5-8-1}
	\end{align}
As $m \geq 2$, we have 
	\begin{align*}
		I_n^{(2m)} \sim \frac{1}{ t^{(2m)}_n}  \left(1+\sum^{\infty}_{j=1} \binom{j+m}{j}^{-1}	\right)=O((\log n)^{-m}).
	\end{align*}
	
	Next, we will show \eqref{eq5-6}. For $1\leq \ell \leq m-1$, we have
	\[
		d_j^{(2\ell)} \sim m! \cdot (2\ell-1)!! \cdot j^{-(m-\ell)}(\log j)^{\ell}\qquad (j \to \infty).
	\]
	This  implies that, for $1\leq \ell \leq m-2$,
	\[
		 \frac{ 1 }{t_n^{(2m)}}   
		 		\sum^{n-1}_{j=1} d_j^{(2\ell)}= O\left(\dfrac{1}{t_n^{(2m)}}\right)=O((\log n)^{-m}),
	\]
	and also, for $1\leq \ell \leq m-1$,
	\[
		 \frac{ 1 }{t_n^{(2m)}}   
		 		\sum^{n-1}_{j=1}  \frac{1}{2j} \cdot  d_j^{(2\ell)}   = O\left(\dfrac{1}{t_n^{(2m)}}\right)=O((\log n)^{-m}).
	\]
	Therefore we obtain
	\begin{equation}
		J^{(2m)}_n= \frac{  1    }{t_n^{(2m)}} \sum^{n-1}_{j=1} \binom{2m}{2m-2}
			d_j^{(2m-2)}-1
		+O((\log n)^{-m}).
	\label{eq5-8-2}\end{equation}
The first two terms in the right hand side of \eqref{eq5-8-2} are
	\begin{equation}\label{eq5-8}
		\begin{split}
				\frac{ (n+m-1)! }{ n^m (n-1)! } \cdot \frac{m}{(\log n)^m} \sum^{n-1}_{j=1} \frac{j^m j!}{(j+m)!} \cdot \frac{(\log j)^{m-1}}{j}-1.	\\
		\end{split}
	\end{equation}
Note that
\begin{equation}\label{asymp:cncn'}
\begin{split}
		c_n&:= \frac{(n+m-1)!}{n^m(n-1)!}=  
		 \prod^{m-1}_{k=1} \left(1+\frac{k}{n}\right)
									=1+O\left(\dfrac{1}{n}\right)\qquad (n \to \infty), \\
		c_n'&:= \frac{n^m n!}{(n+m)!} = \prod^{m}_{k=1} \left(1+\frac{k}{n}\right)^{-1}
			=1+O\left(\dfrac{1}{n}\right)\qquad (n \to \infty).
\end{split} 
\end{equation}	
Hence  \eqref{asymp:cncn'} and \eqref{eq:lemma-asymptotic2} in Appendix \ref{app:calculus} imply that \eqref{eq5-8} is
	\[\begin{split}
		&  c_n \frac{m}{(\log n)^m} \sum^{n-1}_{j=1} c_j' \frac{(\log j)^{m-1}}{j} -1\\
		&= c_n \frac{m}{(\log n)^m} \sum^{n-1}_{j=1} (c_j'-1) \frac{(\log j)^{m-1}}{j} 
			+ (c_n-1)\frac{m}{(\log n)^m} \sum^{n-1}_{j=1} \frac{(\log j)^{m-1}}{j} \\
			&\qquad+	\frac{m}{(\log n)^m} \sum^{n-1}_{j=1} \frac{(\log j)^{m-1}}{j}-1\\
			&=O((\log n)^{-m}),
		\end{split}\]	
	as $n\to \infty$. Here we use the fact that $\displaystyle \sum^{\infty}_{j=1}	(c'_j-1)\frac{(\log j)^{m-1}}{j}$ converges.
Taking (\ref{eq5-8-2}) into consideration, we obtain (\ref{eq5-6}). 
\end{proof}

Now, by induction with respect to $m$, we will prove \eqref{eq:ERWCLTRate=1/2Even} in Theorem~\ref{thmHOT:EvenMoments}. By \eqref{eq7-2}, we have
\begin{align*}
L^{(2)}_n &= 1+\sum_{j=1}^{n-1} \dfrac{1}{j+1} = \sum_{j=1}^n \dfrac{1}{j},
\end{align*}
and
\begin{align*}
M^{(2)}_n &= \dfrac{L^{(2)}_n}{t^{(2)}_n} - 1 = \dfrac{1}{\log n} \left(\sum_{j=1}^n \dfrac{1}{j}-\log n\right) \sim \gamma (\log n)^{-1} \qquad (n\to\infty).
\end{align*}
Let $m\geq 2$ be fixed and suppose that for any $1\leq \ell \leq m-1$
\[
	M^{(2\ell)}_{n}\sim \ell \gamma (\log n)^{-1} \qquad (n\to\infty)
\]
holds true. Then we have that, as $n\to \infty$, 
\[\begin{split}
K^{(2m)}_n&\sim   \frac{   1 }{t_n^{(2m)}}  \sum^{n-1}_{j=1}   \binom{2m}{2m-2}   
					d_j^{(2m-2)}M_j^{(2m-2)}\\
			&\sim	 \frac{ m(m-1)\gamma }{ (\log n)^m} \sum^{n-1}_{j=1} \frac{(\log j)^{m-2}}{j}\\
			&\sim  m\gamma (\log n)^{-1}.
\end{split}\]
Here we used \eqref{eq:lemma-asymptotic2} in Appendix \ref{app:calculus}. In view of Lemma \ref{lem5-2}, $M_n^{(2m)} \sim K_n^{(2m)}$ as $n \to \infty$. This completes the proof of \eqref{eq:ERWCLTRate=1/2Even} in Theorem~\ref{thmHOT:EvenMoments}.
\qed

\appendix

\section{Lemmas from calculus} \label{app:calculus}



\begin{lemma}[Stolz-Ces\`{a}ro: See e.g. Knopp \cite{Knopp56Dover}, p.34] \label{lem:StolzCesaro} If a real sequence $\{a_n\}$ and a positive sequence $\{b_n\}$ satisfy
\[ \lim_{n\to \infty} \dfrac{a_n}{b_n} = L \in \mathbb{R} \cup \{ \pm \infty \},\quad\mbox{and}\quad \sum_{n=1}^{\infty} b_n = +\infty, \]
then
\[ \lim_{n\to \infty} \dfrac{\sum_{k=1}^n a_k}{\sum_{k=1}^n b_k} = L. \]
\end{lemma}


\begin{lemma}[cf.  Knopp \cite{Knopp56Dover}, p.64] \label{Knopp56p64Thm3bis} Let $f(t)$ be a positive function on $[1,\infty)$, which is non-increasing on $[n_0,\infty)$ for some $n_0 \in \mathbb{N}$. For each $n \in \mathbb{N}$, we set
\begin{align*}
S_n:= \sum_{j=1}^n f(j) \quad \mbox{and} \quad I_n:=\int_1^n f(t)\,dt.
\end{align*}
If $I_n \to +\infty$ as $n \to \infty$, then
\begin{align*}
 \dfrac{S_n}{I_n}  = 1+O((I_n)^{-1}) \qquad \mbox{as $n \to \infty$.}
\end{align*}
\end{lemma}
\begin{proof} For $n \geq n_0$, let
\begin{align*}
S'_n := \sum_{j=n_0}^n f(j)  \quad \mbox{and} \quad I'_n := \int_{n_0}^n f(t)\,dt.
\end{align*}
As in p.64 of \cite{Knopp56Dover}, we can see that $0 \leq S'_n - I'_n \leq f(n_0)$ for any $n \geq n_0$. Since
\[ \dfrac{S_n}{I_n} - 1 = \dfrac{S_{n_0-1} - I_{n_0}}{I_n} + \dfrac{S'_n - I'_n}{I_n} \qquad (n \geq n_0), \]
we obtain the conclusion.
\end{proof}

Let $m \geq 1$. The function
\[ f(t) =\frac{(\log t)^{m-1}}{t} \qquad (t \geq 1) \]
is non-increasing on $[e^{m-1},\infty)$. As
\[ I_n=\int^n_1 f(t)\,dt=\frac{(\log n)^m}{m}, \]
Lemma \ref{Knopp56p64Thm3bis} implies that 
\begin{align}\label{eq:lemma-asymptotic2}
\frac{m}{(\log n)^m} \displaystyle \sum^n_{j=1}\frac{ (\log j)^{m-1}}{j}  &=1+O( (\log n)^{-m}) \qquad (n\to \infty).
\end{align}


\section{A remark on Corollary 1 in \cite{FanHuMa21}} \label{appendix:FanHuMa21Cor1}
For readers' convenience, we briefly explain how to obtain \eqref{eq:FanHuMa21Cor1improved} from the bounds given in Corollary 1 in \cite{FanHuMa21}. Let 
\begin{align*}
s_n^2 &:= \sum_{i=1}^n \dfrac{\Gamma(i)^2}{\Gamma(i+\alpha)^2},
\intertext{and}
\sigma_n^2 &:= 
\begin{cases}
\dfrac{\Gamma(n)^2}{\Gamma(n+\alpha)^2} \cdot \dfrac{n}{1-2\alpha} &(-1<\alpha<1/2), \\[4mm]
\dfrac{\Gamma(n)^2}{\Gamma(n+\alpha)^2}  \cdot n\log n &(\alpha=1/2).
\end{cases}
\end{align*}
It is enough to show that
\begin{align}
\label{ineq:DanHuMaCor1Remainder}
\begin{split}
\left| \dfrac{\sqrt{s_n^2}}{\sqrt{\sigma_n^2}} - 1 \right| \leq
\begin{cases}
\dfrac{C'_{\alpha}}{n} & (-1<\alpha< 0), \\[4mm]
\dfrac{C'_{\alpha}}{n^{1-2\alpha}} & (0< \alpha<1/2), \\[4mm]
\dfrac{C'_{1/2}}{\log n}  & (\alpha=1/2).
\end{cases}
\end{split} 
\end{align}
To this end, we use Wendel's inequality \cite{Wendel48}: 
\[ \left(\dfrac{x}{x+\alpha}\right)^{1-\alpha} \leq \dfrac{\Gamma(x+\alpha)}{x^{\alpha} \Gamma(x)} \leq 1 \quad \mbox{for $\alpha \in (0,1)$ and $x>0$.}\]
For $n=1,2,\cdots$, we have
\[ n^{-2\alpha} \leq \dfrac{\Gamma(n)^2}{\Gamma(n+\alpha)^2} \leq n^{-2\alpha} \left(1+\dfrac{\alpha}{n}\right)^{2(1-\alpha)}=n^{-2\alpha}(1+r_n), \]
where $r_n=O(n^{-1})$ as $n \to \infty$. Using this and Lemmas in Appendix \ref{app:calculus}, one can see that $\{s_n^2-\sigma_n^2\}$ is bounded. Noting that 
\begin{align*}
\left| \dfrac{\sqrt{s_n^2}}{\sqrt{\sigma_n^2}} - 1 \right| = \dfrac{|s_n^2 - \sigma_n^2|}{\sqrt{\sigma_n^2} \cdot (\sqrt{s_n^2}+\sqrt{\sigma_n^2})} \leq \dfrac{|s_n^2 - \sigma_n^2|}{\sigma_n^2},
\end{align*}
we obtain \eqref{ineq:DanHuMaCor1Remainder} for $\alpha \in (0,1/2]$. When $\alpha \in (-1,0)$, we use 
\begin{align*}
\left(\dfrac{x-1}{(x-1)+(\alpha+1)}\right)^{1-(\alpha+1)} \leq \dfrac{\Gamma((x-1)+(\alpha+1))}{(x-1)^{\alpha+1} \Gamma(x-1)} \leq 1, 
\end{align*}
that is,
\begin{align*}
\left(\dfrac{x-1}{x+\alpha}\right)^{-\alpha} \leq \dfrac{\Gamma(x+\alpha)}{(x-1)^{\alpha} \Gamma(x)} \leq 1
\end{align*}
for $x>1$. For $n=2,3,\cdots$, we have
\[ n^{-2\alpha}(1+r'_n) \leq \dfrac{\Gamma(n)^2}{\Gamma(n+\alpha)^2}  \leq n^{-2\alpha}(1+r''_n), \]
where $r'_n,r''_n=O(n^{-1})$ as $n \to \infty$. Using this, one can see that there is a positive constant $M_{\alpha}$ such that $|s_n^2-\sigma_n^2| \leq M_{\alpha}n^{-2\alpha}$ for all $n$. As $\sigma_n^2 = O(n^{1-2\alpha})$ as $n \to \infty$, we obtain \eqref{ineq:DanHuMaCor1Remainder} for $\alpha \in (-1,0)$.




\end{document}